\documentclass[11pt,oneside,letterpaper]{amsart}
\usepackage{amssymb}

\usepackage{amsfonts,amsmath,amstext,amsbsy,euscript,amsthm}
\usepackage{mathtools}
\usepackage{enumerate}
\usepackage{mathabx}
\usepackage[all]{xy}
\usepackage[utf8]{inputenc} 
\usepackage{color}
\usepackage{tikz-cd}
\usepackage{mathrsfs}
\usepackage[colorlinks]{hyperref}
\usepackage{pstricks}
\usepackage{graphicx,pstricks,pst-plot,pst-math}

\theoremstyle{plain}
\newtheorem{theorem}{Theorem}[section]
\newtheorem{proposition}[theorem]{Proposition}
\newtheorem{corollary}[theorem]{Corollary}
\newtheorem{lemma}[theorem]{Lemma}

\theoremstyle{definition}
\newtheorem{definition}[theorem]{Definition}

\theoremstyle{remark}
\newtheorem{remark}[theorem]{Remark}

\def\sideremark#1{\ifvmode\leavevmode\fi\vadjust{\vbox to0pt{\vss \hbox to 0pt{\hskip\hsize\hskip1em\vbox{\hsize2cm\tiny\raggedright\pretolerance10000  \noindent #1\hfill}\hss}\vbox to8pt{\vfil}\vss}}}

\numberwithin{equation}{section}
\title{The Snapshot Problem for the Euler-Poisson-Darboux Equation}

\author{Jue Wang}
\address{College of Science, North China Institute of Science \& Technology, Langfang, Hebei, 065201, China}
\email{wangjue\textunderscore math@126.com}

\author{Tomoyuki Kakehi}
\address{Department of Mathematics,
University of Tsukuba,
Tsukuba, Ibaraki, 305-8571, Japan}
\email{kakehi@math.tsukuba.ac.jp}

\author{Fulton Gonzalez}
\address{Department of Mathematics,
Tufts University, 
Medford, MA 02155, USA}
\email{fulton.gonzalez@tufts.edu}

\author{Jens Christensen}
\address{Department of Mathematics, 
Colgate University,
Hamilton, NY 13346, USA}
\email{jchristensen@colgate.edu}

\date{\today}
\subjclass[2020]{Primary: 35L05, Secondary: 58J45, 43A85}

\begin{document}

\maketitle

\begin{abstract}
The generalized Euler-Poisson-Darboux (EPD) equation with complex parameter $\alpha$ is given by
$$
\Delta_x u=\frac{\partial^2 u}{\partial t^2}+\frac{n-1+2\alpha}{t}\,\frac{\partial u}{\partial t},
$$
where $u(x,t)\in \mathscr E(\mathbb R^n\times \mathbb R)$,  with $u$ even in $t$.  For $\alpha=0$ and $\alpha=1$ the solution $u(x,t)$ represents a mean value over spheres and balls, respectively, of radius $|t|$ in $\mathbb R^n$.  In this paper we consider  existence and uniqueness results for the following two-snapshot problem: for fixed positive real numbers $r$ and $s$ and smooth functions $f$ and $g$ on $\mathbb R^n$, what are the conditions under which there is a solution $u(x,t)$ to the generalized EPD equation such that $u(x,r)=f(x)$ and $u(x,s)=g(x)$?  The answer leads to a discovery of Liouville-like numbers related to  Bessel functions, and we also study the properties of such numbers.
\end{abstract}
%
%
%
%
%
%
%
%
%
%
%

%



\section{Introduction}
For a fixed real number $t\geq 0$, let $M^t$ denote the spherical mean value operator on $\mathbb R^n$, which averages functions  over 
spheres of radius $t$:
$$
M^tf(x)=\frac{1}{\Omega_n}\,\int_{S^{n-1}} f(x+t\omega)\,d\omega.
$$
In the above, $\Omega_n=2\pi^{n/2}/\Gamma(n/2)$ is the area of the unit sphere $S^{n-1}$, $d\omega$ is the area measure on $S^{n-1}$, and $f$ is (say) a continuous function or distribution on $\mathbb R^n$.

It can be shown that $M^t$ is surjective as a linear operator on the locally convex spaces $\mathscr E(\mathbb R^n)$ and $\mathscr D'(\mathbb R^n)$. (See, for example, (\cite{Lim} or \cite{CGKW2024}.) However, the kernels of $M^t$ on these spaces are easily seen to be infinite-dimensional (\cite{Ehrenpreis1970}).  

We are interested in the following two-radius mean value existence problem.  Fix positive real numbers $r$ and $s$, with $r\neq s$.  Suppose that $g$ and $h$ are functions in $\mathscr E(\mathbb R^n)$, we wish to determine
whether there exists a
function $f\in\mathscr E(\mathbb R^n)$ such that 
\begin{equation}\label{E:MV-existence1}
g=M^r f\quad\text{and}\quad h=M^sf.
\end{equation}
For almost all choices of $(r,s)$, such a function $f$, if it exists, is unique. This is the two-radius Pompeiu problem.  (See Zalcman's paper \cite{Zalcman1972} for a proof.) 

The two-radius existence problem is clearly overdetermined, so in order for a solution to exist, there must be a relation between the functions $g$ and $h$. This is easy to determine: in fact the mean value operators $M^r$ and $M^s$ are convolution operators with compactly supported distributions, and hence they commute. Thus  for a function $f$ to satisfy \eqref{E:MV-existence1}, $g$ and $h$ must satisfy the \emph{compatibility condition}
\begin{equation}\label{E:compatibility1}
M^s g= M^r h.
\end{equation}
We therefore wish to determine whether the condition \eqref{E:compatibility1} is also sufficient for $f$ to exist.  Considering Zalcman's two-radius Pompeiu result, it is natural to expect
that the answer will depend on the choice of $r$ and $s$.
We can thus refine our two-radius existence problem to ask for the conditions on $r$ and $s$ so that the compatibility condition \eqref{E:compatibility1} on any $g$ and $h$ is sufficient
to guarantee the existence of a function $f$ satisfying \eqref{E:MV-existence1}.

One of the main results of this paper will be to show that for almost all $(r,s)$, the  condition \eqref{E:compatibility1} on any $g$ and $h$ in $\mathscr E(\mathbb R^n)$
is indeed enough to guarantee the existence of an $f\in\mathscr E(\mathbb R^n)$ satisfying \eqref{E:MV-existence1}.  We shall see, however, that there are exceptional pairs $(r,s)$.

We can expand our point of view a bit, and consider the two-radius problem in the context of differential equations. Suppose that $u(x,t)\in C^\infty(\mathbb R^n\times\mathbb R)$ such that $u(x,t)=u(x,-t)$ for all $x$ and $t$.  An easy application of \'Asgeirsson's mean value theorem (\cite{JohnPlaneWaves}, Ch.~V) shows that $u(x,t)=M^t f(x)$ for some function $f\in \mathscr E(\mathbb R^n)$ if and only if $u(x,t)$ solves the initial value problem for the Euler-Poisson-Darboux (EPD) equation,
\begin{align}\label{E:EPD1}
&\Delta u=\frac{\partial^2 u}{\partial t^2} +\frac{n-1}{t}\,\frac{\partial u}{\partial t},\\
&u(x,0)=f(x).\label{E:cauchy1}
\end{align}
  The two-radius problem can then be restated in the following way.  What are the conditions on $r$ and $s$ so that for any $g$ and $h$ satisfying the compatibility condition \eqref{E:compatibility1}, there is a $C^\infty$ solution $u(x,t)$ to the EPD equation \eqref{E:EPD1} that is even in $t$, and such that $u(x,r)=g(x)$ and $u(x,s)=h(x)$ for all $x\in\mathbb R^n$?  And when is this solution $u(x,t)$ unique?  In this context, it will be convenient to think of the parameter $t$ as representing time, and we will call the functions $g$ and $h$ the \textit{snapshots} of the function $u(x,t)$ at the times $r$ and $s$, respectively.
  
Let us note that we can also consider the analogous problem where we take mean values over balls  of fixed radii $r$ and $s$ in $\mathbb R^n$ instead of over spheres.  

In this paper, we will consider two-snapshot existence and uniqueness problems for a more general version of the EPD equation \eqref{E:EPD1} which depends on a complex parameter $\alpha$, and which includes as special cases the existence and uniqueness questions for the two-radius mean value problems mentioned above.  These  generalized EPD equations have a long pedigree, but for our purposes we would like to specifically point out the paper by Bresters (\cite{Bresters1973}), which obtains an explicit fundamental solution in terms of convolution operators whose kernels have Fourier transforms which are normalized Bessel functions.  We would also like to mention Rubin's paper \cite{Rubin2023} which provides some background as well as numerous references to various aspects of the generalized EPD equation.  Finally, we would like to refer the reader to the papers by Lions and Delsarte \cite{DelsarteLions1957,Lions1956, Lions1959}, which guarantee the existence and uniqueness of solutions to the Cauchy problem associated with the generalized EPD equation.

 In Section~\ref{sec:generalizedEPD} we discuss the generalized EPD equation and the solutions found by Weinstein and Bresters under certain decay assumptions on the initial values. These decay assumptions are removed in Section~\ref{sec:solutionEPD}, 
where we only assume smoothness of the Cauchy data. These solutions are given by convolution operators with compactly supported distributions whose Fourier transforms are normalized Bessel functions. The asymptotic properties of Bessel functions are reviewed in Section~\ref{sec:Besselfunctions} where we also introduce Liouville-type numbers related to the zeros of the Bessel functions.
Some of the properties of these numbers are addressed in Section~\ref{Sec:BLproperties}.
In Section~\ref{sec:besselmultipliers} we uncover some of the mapping properties of convolution operators having kernels whose  Fourier transforms are given by these normalized Bessel functions. Finally Section~\ref{sec:twosnapshot} contains the main results of this paper. In it we
introduce the two-snapshot problem for the EPD equation and address the question of existence and uniqueness of solutions. We show that the existence
depends on whether or not the ratio of the snapshot times is of Liouville-type or not.

The authors would like to thank Professor B. Rubin for introducing us to the problem and for providing us with inspiration.

\section{The Generalized Euler-Poisson-Darboux Equation}\label{sec:generalizedEPD}
In what follows we assume that $n\geq 2$.

For a function $u(x,t)$ in $C^\infty(\mathbb R^n\times\mathbb R^+)$, let  us consider the generalized EPD equation with arbitrary complex parameter $\alpha$, 
\begin{equation}\label{E:EPDgen}
\Delta_x u=\frac{\partial^2 u}{\partial t^2}+\frac{n-1+2\alpha}{t}\,\frac{\partial u}{\partial t}
\end{equation}
The cases $\alpha=0$ and $\alpha=(1-n)/2$ specialize to the classical EPD and wave equation on $\mathbb R^n$, respectively.  Let us suppose that $u$ satisfies the following initial conditions:
\begin{equation}\label{E:EPDCauchy}
\lim_{t\to 0^+} u(x,t)=f(x),\quad \lim_{t\to 0^+} \frac{\partial u}{\partial t}\,(x,t)=0,
\end{equation}
where $f\in\mathscr E(\mathbb R^n)$.
As mentioned earlier, equation \eqref{E:EPDgen} has been studied in several different contexts.  Explicit solutions were obtained by Weinstein \cite{Weinstein54} and Bresters \cite{Bresters1973} by considering Fourier transforms in the space variables; see also \cite{Rubin2023}.  For $\alpha$ real and  $\geq (1-n)/2$ (and under slightly different differentiability assumptions on $u(x,t)$) the unique solution they obtained is given by the convolution
\begin{equation}\label{E:EPDconvolutionsoln}
u(x,t)=(f*m_\alpha^t)(x),
\end{equation}
where $m^t_\alpha\in\mathscr E'(\mathbb R^n)$ is given by
\begin{equation}\label{E:FS1}
m^t_\alpha(y)=\frac{\Gamma(\alpha+n/2)}{\pi^{n/2}\,\Gamma(\alpha)}\,t^{-n}\,(1-\|y/t\|^2)^{\alpha-1}_+.
\end{equation}
Here $m_t^\alpha$ acts on test functions $\varphi$ by
\begin{equation}\label{E:FS2}
m_\alpha^t(\varphi)=\frac{\Gamma(\alpha+n/2)}{\pi^{n/2}\,\Gamma(\alpha)}\,\left(t^{2-n-2\alpha}\right)\,\int_{\|x\|\leq t} \varphi(x)\,\left(t^2-\|x\|^2\right)^{\alpha-1}\,dx.
\end{equation}
When $\alpha=1$, note that $m_1^t(\varphi)$ is the average of $\varphi$ over the ball of radius $t$ centered at the origin.

Let us clarify the meaning of \eqref{E:FS2} by extending $\alpha$ to a complex parameter. 
 For fixed $t$ and $\varphi$, the integral on the right hand side of \eqref{E:FS2} 
\begin{equation}\label{E:dist-int}
\int_{\|x\|\leq t} \varphi(x)\,\left(t^2-\|x\|^2\right)^{\alpha-1}\,dx
\end{equation}
converges and is a holomorphic function of $\alpha$ on the half plane 
$\text{Re}\,\alpha>0$.  It can be analytically continued to be a meromorphic function on $\mathbb C$, with simple poles at $\alpha=0,-1,-2,\ldots$.  To see this, we write the integral \eqref{E:dist-int} in polar coordinates to obtain an Erdelyi-Kober type fractional integral:
\begin{align}
\int_{\|x\|\leq t} \varphi(x)\,\left(t^2-\|x\|^2\right)^{\alpha-1}\,dx&=\Omega_n\,\int_0^t (M^s\varphi)(0)\,(t^2-s^2)^{\alpha-1}\,s^{n-1}\,ds\nonumber \nonumber\\
&=\Omega_n\,t^{n+2\alpha-2}\,\int_0^1 (M^{tu}\varphi)(0)\,(1-u^2)^{\alpha-1}\,u^{n-1}\,du\nonumber\\
&=\Omega_n\,t^{n+2\alpha-2}\,\left(u^{n-1}\,(1-u^2)^{\alpha-1}_+\right) (M^{tu}(\varphi)(0))\label{E:dist-int2}
\end{align}
In the above, $\Omega_n=2\pi^{n/2}/\Gamma(n/2)$ is the area of the unit sphere in $\mathbb R^n$, and $(1-u^2)^{\alpha-1}_+$ is the distribution on $\mathbb R$ given for $\text{Re}\,\alpha>0$ by
$$
(1-u^2)^{\alpha-1}_+(\psi)=\int_0^1 (1-u^2)^{\alpha-1}\,\psi(u)\,du,\qquad \psi\in\mathscr E(\mathbb R).
$$

The right hand side of \eqref{E:dist-int2} extends to a meromorphic function on $\mathbb C$ with simple poles at $\alpha=0,-1,-2,\ldots$.  In particular, one can see that the residue of the pole at $\alpha=0$ equals $(\Omega_n/2)\,t^{n-2}\,(M^t\varphi)(0)$.

These poles of the right hand side of \eqref{E:dist-int2} are offset by the poles of $\Gamma(\alpha)$ in \eqref{E:FS2}, so it follows that $\alpha\mapsto m_\alpha^t(\varphi)$ is a meromorphic
function on $\mathbb C$ with simple poles at $\alpha\in -n/2-\mathbb Z^+$.  When $\alpha=0$, this gives us $m_0^t(\varphi)=M^t\varphi(0)$.     

For fixed $t$, the map $\alpha\mapsto m_\alpha^t$ is thus a meromorphic function on $\mathbb C$ with values in $\mathscr E'(\mathbb R^n)$, and more specifically in the space of distributions supported in the ball $\|x\|\leq t$.  (If $\varphi\in\mathscr E(\mathbb R^n)$ has support outside the ball $\|x\|\leq t$, then $m^t_\alpha(\varphi)=0$ when $\text{Re}\,\alpha>0$, so it equals $0$ for all $\alpha$ by analyticity.)  Here we are endowing $\mathscr E'(\mathbb R^n)$ with the weak* topology.

Combining \eqref{E:FS2} and \eqref{E:dist-int2} we have
\begin{equation}\label{E:mt-polar}
m^t_\alpha(\varphi)=\frac{\Gamma(\alpha+n/2)}{\pi^{n/2}\,\Gamma(\alpha)}\,\Omega_n\,\left(u^{n-1}\,(1-u^2)^{\alpha-1}_+\right) (M^{tu}(\varphi)(0)),
\end{equation}
for all $\alpha\notin -n/2-\mathbb Z^+$.  
If we now fix $\alpha$ in this set, \eqref{E:mt-polar} shows that $m_\alpha^t(\varphi)$ extends to an even $C^\infty$ function of $t$.  In particular, one can readily verify that
\begin{equation}\label{E:init-data}
m_\alpha^0(\varphi)=\varphi(0)\quad\text{and}\quad \left.\frac{d}{dt}\left(m^t_\alpha(\varphi)\right)\right|_{t=0}=0.
\end{equation}

\section{An Elementary Solution of the Generalized EPD Equation}
\label{sec:solutionEPD}

As mentioned earlier, the initial value problem \eqref{E:EPDgen}--\eqref{E:EPDCauchy} was solved in \cite{Bresters1973} for real $\alpha$ by taking the Fourier transform of the EPD equation with respect to the space variables and then by solving the resulting ODE.  Of course, this method initially assumes that $u(x,t)$ has certain decay properties with respect to $x$ for fixed $t$.  In this section, we obtain the solution \eqref{E:EPDconvolutionsoln} through an explicit calculation, which imposes no decay conditions on the space variables $x$, and which relies on Lions' uniqueness theorem for solutions of such equations (\cite{Lions1959}, Th\'eor\`eme 4.1). In the case when $\alpha$ is real and $\geq (1-n)/2$ we also present an alternative solution method which uses an analogue of Asgeirsson's mean value theorem.


With this goal in mind, we begin with the following result.
\begin{lemma}\label{T:EPD1}
Fix $\alpha\in\mathbb C,\;\alpha\notin -n/2-\mathbb Z^+$.  Then we have
\begin{equation}\label{E:EPD-dist}
\Delta\, m^t_\alpha =\left(\frac{\partial^2}{\partial t^2}+\frac{n-1+2\alpha}{t}\,\frac{\partial}{\partial t}\right)\,m_\alpha^t.
\end{equation}
\end{lemma}
\begin{proof}  It is enough to prove \eqref{E:EPD-dist} for the case $\text{Re}\,\alpha>1$, the general case following by analytic continuation.  For this case, and in view of \eqref{E:FS2}, this amounts to proving that
\begin{multline}\label{E:MV-alpha}
t^{2-n-2\alpha}\,\int_{\|x\|\leq t} \Delta\varphi(x)\,\left(t^2-\|x\|^2\right)^{\alpha-1}\,dx=\\
=\left(\frac{\partial^2}{\partial t^2}+\frac{n-1+2\alpha}{t}\,\frac{\partial}{\partial t}\right)\,\left(t^{2-n-2\alpha}\,\int_{\|x\|\leq t} \varphi(x)\,\left(t^2-\|x\|^2\right)^{\alpha-1}\,dx \right).
\end{multline}
for any $\varphi\in\mathscr E(\mathbb R^n)$.  Taking into account the first equality in \eqref{E:dist-int2} (and factoring out $\Omega_n$) we therefore need to prove that
\begin{multline*}
t^{2-n-2\alpha}\,\int_0^t M^s(\Delta\varphi)(0)\,(t^2-s^2)^{\alpha-1}\,s^{n-1}\,ds=\\
=\left(\frac{\partial^2}{\partial t^2}+\frac{n-1+2\alpha}{t}\,\frac{\partial}{\partial t}\right)\,\left(t^{2-n-2\alpha}\,\int_0^t (M^s\varphi)(0)\,(t^2-s^2)^{\alpha-1}\,s^{n-1}\,ds\right).
\end{multline*}
If we put $F(s)=(M^s\varphi)(0)$, then  $M^s(\Delta\varphi)(0)=F''(s)+((n-1)/s)\,F'(s)$, so upon using the substitution $s=tu$, the above equality reduces to
\begin{multline*}
\int_0^1\left(F''(tu)+\frac{n-1}{tu}\,F'(tu)\right)\,(1-u^2)^{\alpha-1}\,u^{n-1}\,du=\\
=\left(\frac{\partial^2}{\partial t^2}+\frac{n-1+2\alpha}{t}\,\frac{\partial}{\partial t}\right)\,\int_0^1 F(tu)\,(1-u^2)^{\alpha-1}u^{n-1}\,du.
\end{multline*}
A straightforward calculation using integration by parts then shows that both sides are equal to
$$
\frac{2\alpha-2}{t}\,\int_0^1 F'(tu)\,(1-u^2)^{\alpha-2}\,u^n\,du,
$$
proving \eqref{E:EPD-dist}.
\end{proof}

It is clear from Lemma \ref{T:EPD1} and the relation \eqref{E:init-data} that $u(x,t)=f*m_\alpha^t$ is a solution of the Cauchy problem \eqref{E:EPDgen}--\eqref{E:EPDCauchy}, for $\alpha\notin -n/2-\mathbb Z^+$. 

Now according to Lions (\cite{Lions1959}, Th\'eor\`eme 4.1; see also \cite{Lions1956}), it is precisely for such $\alpha$ that any solution to the Cauchy problem  is unique.  This gives us the following result.

\begin{theorem}\label{T:EPDsoln} Assume that $\alpha\notin -n/2-\mathbb Z^+$.
If $f\in\mathscr E(\mathbb R^n)$, then $u(x,t)=f*m_\alpha^t$ is the unique solution to the Cauchy problem for the generalized EPD equation \eqref{E:EPDgen}--\eqref{E:EPDCauchy}. 
\end{theorem}

The above solution was obtained for real $\alpha$ by Bresters \cite{Bresters1973} using Fourier transform methods, under the assumption that $f$ is rapidly decreasing on $\mathbb R^n$.  See also the paper by Weinstein \cite{Weinstein54}.



For completeness, we offer below an alternative proof of Theorem~\ref{T:EPDsoln}, in the case when $\alpha\geq (1-n)/2$, which uses a variant of Asgeirsson's mean value theorem with very few special tools.  


We note of course that when $\alpha=(1-n)/2$, \eqref{E:EPDgen}--\eqref{E:EPDCauchy} reduces to the Cauchy problem for the classical wave equation, whose solution is unique and well-known. 

In any case, let us suppose that $\text{Re}\,\alpha\geq (1-n)/2$ and that $u(x,t)$ is a $C^\infty$ function on $\mathbb R^n\times\mathbb R$, even in $t$, and is a solution to the generalized EPD equation \eqref{E:EPDgen}.  Fix $x\in\mathbb R^n$, and for $s$ and $t$ real let
\begin{equation}\label{E:mvalpha1}
U(s,t)=(u(\;\cdot\;, t)*m_\alpha^s)(x),
\end{equation}
where the convolution is taken with respect to the first argument.  By \eqref{E:mt-polar} and the subsequent remarks, $U(s,t)$ is a $C^\infty$ function of $(s,t)$ in $\mathbb R\times\mathbb R$, is even in $s$ and in $t$, and with $U(0,t)=u(x,t)$.  According to Lemma \ref{T:EPD1}, $U(s,t)$ satisfies the ``ultrahyperbolic'' differential equation
$$
\left(\frac{\partial^2}{\partial s^2}+\frac{n-1+2\alpha}{s}\,\frac{\partial}{\partial s}\right)\,U(s,t)=
\left(\frac{\partial^2}{\partial t^2}+\frac{n-1+2\alpha}{t}\,\frac{\partial}{\partial t}\right)\,U(s,t).
$$
If we put $A(s)=s^{n-1+2\alpha}$, then the above equation can be written as
\begin{equation}\label{E:asg1}
\left(\frac{\partial^2}{\partial s^2}+\frac{A'(s)}{A(s)}\,\frac{\partial}{\partial s}\right)\,U(s,t)=
\left(\frac{\partial^2}{\partial t^2}+\frac{A'(t)}{A(t)}\,\frac{\partial}{\partial t}\right)\,U(s,t).
\end{equation}
for all $(s,t)$ in the first quadrant.  

We wish to prove that $U(s,t)=U(t,s)$ for all $s$ and $t$.  It is enough to prove this for $(s,t)$ in the (closed) first quadrant.  For this, we can reprise part of the proof of \'Asgeirsson's mean value theorem (see \cite{Asg37} or \cite{GGA}, Ch.II).

Let $V(s,t)=U(s,t)-U(t,s)$.  Then, since \eqref{E:asg1} is symmetric in $s$ and $t$, $V(s,t)$ satisfies the equation
$$
\left(\frac{\partial^2}{\partial s^2}+\frac{A'(s)}{A(s)}\,\frac{\partial}{\partial s}\right)\,V(s,t)=
\left(\frac{\partial^2}{\partial t^2}+\frac{A'(t)}{A(t)}\,\frac{\partial}{\partial t}\right)\,V(s,t).
$$
Multiplying both sides by $2 A(s)\,\partial V/\partial t$, we obtain the equality
\begin{equation}\label{E:ultrahyp}
\text{div}\,\left\langle -2A(s)\,\frac{\partial V}{\partial s}\,\frac{\partial V}{\partial t},A(s)\,\left(\left(\frac{\partial V}{\partial s}\right)^2+\left(\frac{\partial V}{\partial t}\right)^2\right)\right\rangle = -2A(s)\,\frac{A'(t)}{A(t)}\,\left(\frac{\partial V}{\partial t}\right)^2.
\end{equation}
Consider the region $\mathcal R$ in the first quadrant 
 of the $(s,t)$-plane enclosed by the triangle $OPQ$, where $OP$ belongs to the line $s=t$ and $PQ$ is horizontal.  We integrate both sides of \eqref{E:ultrahyp} over $\mathcal R$ and apply the divergence theorem to the left hand integral.  Note that the right hand integral will converge absolutely since 
$$
\left|A(s)\,\frac{A'(t)}{A(t)}\right|=|n-1+2\alpha|\,s^{n-1+2\,\text{Re}\,\alpha}\,t^{-1}\leq |n-1+2\alpha|\, t^{n-2+2\,\text{Re}\,\alpha}
$$
since $n-1+2\,\text{Re}\,\alpha\geq 0$, and $s\leq t$ inside the triangle.  Replacing $(s,t)$ by polar coordinates then proves the absolute convergence.


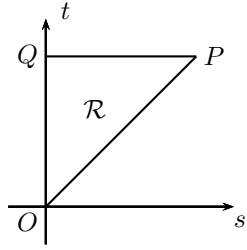
\begin{center}
\begin{pspicture}(-.5,-.5)(2.5,2.5)
\psline(0,2)(2,2)
\psline(0,0)(2,2)
\rput[l](2.5,-.2){\small $s$}
\rput[l](0.2,2.6){\small $t$}
\rput[r](-.1,-.2){\small $O$}
\rput[l](2.1,2){\small $P$}
\rput[r](-.1,2){\small $Q$}
\rput[l](.5,1.3){\small $\mathcal R$}
\psaxes[labels=none,ticks=none]{->}(0,0)(-.5,-.5)(2.5,2.5)
\end{pspicture}
\end{center}
The integral of the right hand side of \eqref{E:ultrahyp} is $\leq 0$ since $n-1+2\alpha>0$. On the other hand, the line integral on the left is $\geq 0$.  In fact, the line integral over $OP$ vanishes since $\partial V/\partial s+\partial V/\partial t=0$ on the line $s=t$, and the line integral over $OQ$ vanishes since $A(s)=0$ there.  Finally, the line integral over $PQ$ is clearly $\geq 0$.  Since $A'(t)> 0$ in $\mathcal R$, it follows that $\partial V/\partial t=0$ at all points in $\mathcal R$, and, as $V$ vanishes on $OP$, we see that $V=0$ on $\mathcal R$.  Since $\mathcal R$ is arbitrary, we conclude that $V$ is identically zero on the whole plane.  Thus $U(s,t)=U(t,s)$ for all $(s,t)$, so in particular $U(0,t)=U(t,0)$ for all $t$.  By \eqref{E:mvalpha1}, this gives
$$
u(x,t)=(f*m^t_\alpha)(x).
$$

\section{
Bessel functions and their zeros}\label{sec:Besselfunctions}
\subsection{Preliminaries on 
Bessel functions}\label{sec:Besselprelims}

A direct calculation using \eqref{E:mt-polar} for $\text{Re}\,\alpha>0$, plus analytic continuation, shows that  the Fourier-Laplace transform of the distribution $m_\alpha^t$ equals
\begin{equation}\label{E:malphaFT}
 (m_\alpha^t)^\sim(\zeta)=\Gamma(\alpha+n/2)\,j_{(n-2)/2+\alpha}(t\,(\zeta\cdot\zeta)^{1/2}), 
\end{equation}
for $\zeta\in \mathbb C^n$ and $\alpha\notin -n/2 - \mathbb Z^+$.
In the above, for any $\nu\in\mathbb C$, $j_\nu$ denotes the normalized Bessel function of order $\nu$:
\begin{equation}\label{E:normalizedbessel}
j_\nu(z)=(z/2)^{-\nu}\,J_\nu(z).
\end{equation}
(See also \cite{Bresters1973}, equation 3.5.)

We are only interested in the cases when the values of the parameter $\alpha$ ensure that the initial value problem \eqref{E:EPDgen}--\eqref{E:EPDCauchy} has a unique solution; that is to say, when $\alpha\neq -n/2-\mathbb Z^+$. Therefore we will assume henceforth that $\nu$ is not a negative integer.

To solve the two-snapshot existence problem stated in the Introduction, we will now need to collect some basic facts about Bessel functions, using standard references such as \cite{WatsonTreatiseBessel} or \cite{NIST:DLMF}, and specifically facts regarding their asymptotic properties and the distribution of their large zeros.

From the series expansion of $J_\nu$, we have that
\begin{equation}\label{E:besselseries}
j_\nu(z) = \sum_{k=0}^\infty \frac{(-1)^k}{k!\Gamma(k+\nu+1)}\left(\frac{z}{2}\right)^{2k}
\end{equation}
for $z\in\mathbb{C}$.   While $J_\nu(z)$ is multivalued in general, note that $j_\nu(z)$ is well-defined, even, and entire. 

With the possible exception of $z=0$, the zeros of $j_\nu$ coincide with those of $J_\nu$.  The zeros of $j_\nu$ are simple, and they are real when $\nu$ is real and $> -1$ (\cite{WatsonTreatiseBessel}, Sec.~15.27).  According to \cite{WatsonTreatiseBessel}, Sec.~15.53, the large (in modulus) zeros of $j_\nu$ with nonnegative real part, for any complex order $\nu$, are indexed by all sufficiently large positive integers $m$ in the following manner:
\begin{equation}\label{E:largezeros}
   a_m= m\,\pi+\left(\nu-\frac 12\right)\,\frac{\pi}{2}+O\left(\frac 1m\right).
\end{equation}
We observe that these large zeros tend towards the horizontal line $\text{Im}\,z=(\pi/2)\,\text{Im}\,\nu$ and are spaced approximately $\pi$ units apart.

We will also make use of the following asymptotic estimate of $j_\nu(z)$ for large $z$, again for any complex order $\nu$, which comes from the asymptotic expansion of $J_\nu(z)$ due to Hankel:
\begin{multline}\label{E:asymptotic-est}
    j_\nu(z)\;\sim\;\left(\frac{2}{\pi z}\right)^{\frac 12+\nu}\,\left[\cos (z-(\nu+1/2)(\pi/2))\,\left(1+O\left(\frac 1{z^2}\right)\right)\right.\\
    - \frac 1z\,\left.\sin (z-(\nu+1/2)(\pi/2))\,\left(1+O\left(\frac 1{z^2}\right)\right)\right]
\end{multline}
(See \cite{WatsonTreatiseBessel}, Sec.~7.21.)
The above estimate is valid for all large  $z$ with $|\text{arg}\,z|$ bounded away from $\pi$, and in particular for all large $z$ with positive real part.  In the above, the complex exponent $z^\nu$ refers to its principal part.

Using the fact that $\cos(x+iy) = \cos(x)\cosh(y)-i\sin(x)\sinh(y)$, the approximation
\eqref{E:asymptotic-est} shows that $j_\nu$ satisfies the estimate
\begin{equation}\label{E:PW-est1}
    |j_\nu(z)|\leq C\,|z|^{-\text{Re}\,\nu}\,e^{|\text{Im}\,z|},\qquad z\in\mathbb C,
\end{equation}
for some $C>0$.  In particular note that $j_\nu$ is bounded on $\mathbb R$ when $\text{Re}\,\nu\geq 0$.
By the Paley-Wiener Theorem we see that $j_\nu$ is the one-dimensional Fourier transform of a distribution 
supported on the interval $[-1,1]$.   We will return to this distribution in section~\ref{sec:besselmultipliers}.

Finally, we will make use of the easily derived connection formula
\begin{equation}\label{E:Bessel-connection}
    j_\nu'(z)=-z\,j_{\nu+1}(z).
\end{equation}
(See \cite{NIST:DLMF}, Sec. 10.6.)

\subsection{Liouville-type numbers related to 
Bessel functions}
In this section we consider certain complex numbers related to the zeros of the normalized Bessel function $j_\nu$, where $\nu$ is a fixed complex order.

The two-snapshot existence problem for the generalized EPD equation \eqref{E:EPDgen} essentially turns out to be a problem of small denominators \cite{Rubin2023}. In the classical formulation of such problems, Liouville numbers play a crucial role. Liouville numbers are transcendental numbers which can be very closely approximated by rational numbers \cite{BakerTranscendental}.  In the present case we will need to consider complex numbers which can be very closely approximated by ratios of zeros of Bessel functions.
For this we introduce the notions of $j_\nu$-rational and $j_\nu$-Liouville numbers.
\begin{definition}
  A nonzero complex number $z$ is said to be \emph{$j_\nu$-rational} if it is of the form $z=a/b$ where $a$ and
  $b$ are zeros of the Bessel function $j_{\nu}$ and $b\neq 0$.

  A complex number $z$ is called \emph{$j_\nu$-Liouville}  provided that (i) $z$ is not
  $j_\nu$-rational, and (ii) for every $N\in\mathbb{N}$
  there are zeros $a$ and $b$  of the Bessel function $j_{\nu}$, with $|b|>1$,
  such that $|z-a/b|<1/|b|^N$.
\end{definition}
 From \eqref{E:besselseries} it is not hard to see that $j_{1/2}(z)=\frac{\sin(z)}{z}$ and therefore the $j_{1/2}$-Liouville numbers are the same as classical Liouville numbers. (For the definition of a Liouville number see for example \cite{Hardy-Wright} or \cite{Steuding2005}).

It is also easy to see that any $j_\nu$-Liouville number must be a real number.  In fact, if $z$ is $j_\nu$-Liouville, then there are sequences $\{a_\ell\}$ and $\{b_\ell\}$ consisting of zeros of $j_\nu$ such that $z=\lim a_\ell/b_\ell$.  It is clear that $\lim |b_\ell|=\infty$.  Now by \eqref{E:largezeros} the zeros of $j_\nu$ have bounded imaginary part.  Hence
$$
z=\lim_{\ell\to\infty} \frac{\text{Re}\,a_\ell +i\,\text{Im}\,a_\ell}{\text{Re}\,b_\ell+i\,\text{Im}\,b_\ell}
=\lim_{\ell\to\infty} \frac{\text{Re}\,a_\ell}{\text{Re}\,b_\ell}.
$$
More is true: if the index $\nu$ is not real, we will show in Corollary \ref{T:nu-complex} below that there are no $j_\nu$-Liouville numbers. 

Since the large zeros of $j_\nu$ are spaced like an arithmetic sequence, one would expect the distribution of the $j_\nu$-Liouville numbers in $\mathbb C$ to be similar to that of the Liouville numbers in $\mathbb R$.  This will indeed be the case when $\nu$ is real, as we will show in Section~\ref{Sec:BLproperties}.  In particular, we will show that the set of $j_\nu$-Liouville numbers is uncountable, dense and has measure zero in $\mathbb R$. We also note that the reciprocal of a $j_\nu$-Liouville number is $j_\nu$-Liouville; the proof is the same as for standard Liouville numbers. 

\begin{remark}
    The definition of $j_\nu$-Liouville numbers can be considerably generalized.
    For any discrete (and separated) set $S$ of real numbers one can define 
    $S$-rational numbers to be all possible ratios $a/b$ with $a,b$
    in $S$ ($b\neq 0$) while an $S$-Liouville number $x$ is not $S$-rational and for any $N$ there are
    $a,b$ in $S$ such that
    $|x-a/b|<1/|b|^N$.
\end{remark}

\section{Bessel Fourier-multipliers}\label{sec:besselmultipliers}
In this section we look at some of the mapping properties of
convolution operators whose Fourier transforms are normalized Bessel functions.

Fix $\nu\in\mathbb C$.   Now for any $\zeta\in\mathbb C^n$, we have
\begin{equation}\label{E:zeta-estimate}
|\text{Re}\,(\zeta\cdot \zeta)^{1/2}|\leq \|\text{Re}\,\zeta\|,
\quad
|\text{Im}\,(\zeta\cdot \zeta)^{1/2}|\leq \|\text{Im}\,\zeta\|
\end{equation}
(\cite{Zalcman1972}, Sec.~3) and of course 
$|(\zeta\cdot\zeta)^{1/2}|\leq \|\zeta\|$.
Hence from \eqref{E:PW-est1} and the subsequent remark, we obtain
$$
|j_\nu(t(\zeta\cdot\zeta)^{1/2})|\leq C\,e^{|t|\,\|\text{Im}\,\zeta\|},\qquad t\in\mathbb R,\;\zeta\in\mathbb C^n.
$$
when $\text{Re}\,\nu\geq 0$, and
$$
|j_\nu(t(\zeta\cdot\zeta)^{1/2})|\leq C\,\|\zeta\|^{-\text{Re}\,\nu}\,e^{|t|\,\|\text{Im}\,\zeta\|},\qquad t\in\mathbb R,\;\zeta\in\mathbb C^n.
$$
when $
\text{Re}\,\nu<0$.
By the Paley-Wiener theorem, there is a unique distribution $S^t_\nu$, supported in the closed ball $\overline B_{|t|}(0)$ such that $(S^t_\nu)^\sim(\zeta)=j_\nu(t\,(\zeta\cdot\zeta)^{1/2})$, for all $\zeta\in\mathbb C^n$.
When $\nu=\alpha+(n-2)/2$, we note from \eqref{E:malphaFT} that
$$
S^t_\nu=\frac{1}{\Gamma(\alpha+n/2)}\,m^t_\alpha.
$$
Our immediate objective now is to show that when $t\neq 0$, the convolution operator 
$$
\varphi\mapsto \varphi*S^t_\nu
$$
is a surjective linear operator on $\mathscr E(\mathbb R^n)$.
It turns out that this depends on the fact that the Fourier transform $(S^t_\nu)^\sim$ does not decrease too rapidly on balls in $\mathbb C^n$ centered at points in $\mathbb R^n$.  We formulate  this more precisely as follows.


\begin{definition}
  A function $u:\mathbb{C}^n\to\mathbb{C}$ is called \emph{slowly decreasing}
  provided that there is a constant $A>0$ such that
  \begin{equation}
    \label{eq:slowdecay}
    \sup\{ |u(\zeta)| 
    : \zeta\in \mathbb{C}^n, \|\zeta-\xi\|\leq A\log(2+\| \xi\|)\}
    \geq (A+\|\xi\|)^{-A}
\end{equation}
for all $\xi\in \mathbb{R}^n$.
\end{definition}

The following Theorem is due to Ehrenpreis \cite{Ehrenpreis1960}.
\begin{theorem}\label{T:convsurj}
 Assume that $S \in \mathscr{E}'(\mathbb{R}^n)$.  
  Then the convolution operator on $\mathscr E(\mathbb{R}^n)$ given by $\varphi\mapsto \varphi*S$ is surjective
 if and only if the Fourier transform $\widetilde S$ is slowly decreasing.
\end{theorem}

This allows us to prove the following.
\begin{theorem}\label{T:s-nu-surjectivity}
    For any real number $t\neq 0$ and any $\nu\in\mathbb C$, the convolution operator on $\mathscr E(\mathbb R^n)$ given by $\varphi\mapsto \varphi*{S^t_\nu}$ is surjective.
\end{theorem}

\begin{proof}  In view of the preceding theorem we just need  to prove that the Fourier transform $j_\nu(t(\zeta\cdot\zeta)^{1/2})$ is slowly decreasing.  

We will first show that $j_\nu(z)$ is slowly decreasing as a function of $z\in\mathbb C$.
For this, we will use the asymptotic estimate \eqref{E:asymptotic-est}.  Observing that $|\sin z|\leq 1+|\cos z|$ for any $z\in \mathbb C$,  the estimate shows that it is enough  to prove that
$$
\left(\cos (z-(\nu+1/2)(\pi/2)\right)
$$
is slowly decreasing in $\mathbb C$. But for any real number $x$, we have
\begin{multline*}
\left(\cos (x-(\nu+1/2)(\pi/2)\right)=
 \cos\left(x-(\text{Re}\,\nu+1/2)\pi/2\right)\,\cosh\left((\pi\,\text{Im}\,\nu)/2\right)\\
\qquad\qquad\qquad\qquad\qquad -i\,\sin\left(x-(\text{Re}\,\nu+1/2)\pi/2\right)\,\sinh\left((\pi\,\text{Im}\,\nu)/2\right)
\end{multline*}
Since $\cosh((\pi\,\text{Im}\,\nu)/2)\geq 1$ and $\cosh((\pi\,\text{Im}\,\nu)/2)> |\sinh((\pi\,\text{Im}\,\nu)/2)|$, it is clear that there is a real number $s$ with $|s-x|<\pi/2$ such that the right hand side above has modulus $>1/2$.  

It follows that there is a constant $c>0$ such that for any $x\in\mathbb R$, there is an $s\in\mathbb R$ within $\pi/2$ distance units of $x$ for which
$$
|j_\nu(s)|\geq c\,\left(1+|x|\right)^{-\text{Re}\,\nu-1/2}.
$$
This establishes the slow decrease of $j_\nu(z)$.

Now the argument above in fact shows that for any $\xi\in\mathbb R^n$, there is a point $\xi'\in\mathbb R^n$ within $\pi/(2t)$ distance units of $\xi$ such that $|j_\nu(t\,\|\xi'\|)|\geq c\,\left(1+\|t\,\xi\|\right)^{-\text{Re}\,\nu-1/2}$.  An appropriately large  constant $A$ can be then be found to establish the slow decrease condition \eqref{eq:slowdecay} for $j_\nu(t\,(\zeta\cdot\zeta)^{1/2})$. 
\end{proof}

\begin{remark}\label{rem:surj1}
    By \eqref{E:malphaFT}, the theorem above implies that the convolution operator $\varphi\mapsto \varphi*m^t_\alpha$ on $\mathscr E(\mathbb R^n)$ is surjective whenever $\alpha\notin -n/2-\mathbb Z^+$ and $t\neq 0$.
    This extends some of the results of \cite{CGK2017}, where
    it was shown that spherical mean value operators are surjective
    on $\mathscr{E}(\mathbb{R}^n)$ (this corresponds to $\alpha=0$).
    As mentioned earlier, for  $\alpha=1$, the convolution with $m^1_t$ corresponds to taking mean values over
    balls of radius $|t|$, and these mean value operators are thus also surjective on $\mathscr E(\mathbb R^n)$.  
\end{remark}
\begin{remark}\label{rem:surj2}
The operator $\varphi\mapsto \varphi*S^t_\nu$ on $\mathscr E(\mathbb R^n)$ is not injective.  Indeed, if the two values of $t\,(\zeta\cdot\zeta)^{\frac 12}$ are zeros of $j_\nu$, then the function $\varphi(x)=e^{i(\zeta\cdot x)}$ belongs to its kernel.  
\end{remark}
\begin{corollary}\label{T:one-snapshot}
Suppose that $\alpha\notin -n/2-\mathbb Z^+$, and fix $t_0>0$.  For any $g\in\mathscr E(\mathbb R^n)$,  there is a $C^\infty$ solution $u(x,t)$ to the generalized EPD equation \eqref{E:EPDgen} which is even in $t$, and such that $u(x,t_0)=g(x)$ for all $x\in\mathbb R^n$.
\end{corollary}

This result follows immediately from Theorem \ref{T:EPDsoln}, equation \eqref{E:malphaFT}, and Theorem~\ref{T:s-nu-surjectivity} and the subsequent Remark~\ref{rem:surj1}.  Note that by Remark~\ref{rem:surj2}, this solution $u(x,t)$ is not unique.

\section{The two-snapshot problem for the generalized EPD equation}
\label{sec:twosnapshot}

Corollary~\ref{T:one-snapshot} can be construed as an existence theorem for a solution to the one-snapshot problem associated with the generalized EPD equation \eqref{E:EPDgen}. Let us now consider the two-snapshot problem, as described in the Introduction, which we reprise as follows.  Here we will assume that $\alpha\notin -n/2-\mathbb Z^+$, and as 
usual $\nu=\alpha+(n-2)/2$.

Assume that $r$ and $s$ are nonzero, $r\neq s$ and suppose that $g,\,h\in\mathscr E(\mathbb R^n)$ satisfy the compatibility condition
\begin{equation}\label{E:compatibility2}
g*m^r_\alpha=h*m^s_\alpha
\end{equation}
Does there exist an $f\in\mathscr E(\mathbb R^n)$ such that $f*m^s_\alpha=g$ and $f*m^r_\alpha=h$?  If such an $f$ exists, is it unique?

Since the distributions $m_\alpha^t$ are even in $t$, in what follows we can assume that both $r$ and $s$ are positive.

We first address the uniqueness question.  For this we observe that if $r/s$ is $j_\nu$-rational, there exists a $z\in\mathbb C$ such that $j_\nu(rz)=0$ and $j_\nu(sz)=0$.  For any $\zeta\in\mathbb C^n$ such that $\zeta\cdot\zeta=z^2$, \eqref{E:malphaFT} shows that the function $f(x)=e^{i\,\zeta\cdot x}$ satisfies $f*m^r_\alpha=0$ and $f*m^s_\alpha=0$. 

On the other hand, if $r/s$ is not $j_\nu$-rational, we do have uniqueness.
\begin{theorem}\label{T:2snapshotuniqueness}
Suppose that $r/s$ is not $j_\nu$-rational.  If $f*m^r_\alpha=0$ and $f*m^s_\alpha=0$, then $f=0$.
\end{theorem}
This result is essentially proved in Theorem 3 in Zalcman's two-radius Pompeiu paper \cite{Zalcman1972}, although it is stated for $n=2$ and has a slightly different formulation. For completeness, we supply a proof adapted to our purposes.
\begin{proof}(Zalcman)
Let $\mathbb E'_n$ be the vector space consisting of the Fourier transforms of elements of $\mathscr E'(\mathbb R^n)$.  Then the elements of $\mathbb E'_n$ are precisely the entire functions of exponential type on $\mathbb C^n$ which are polynomially increasing in $\mathbb R^n$.  We topologize $\mathbb E'_n$ so that the Fourier transform is a homeomorphism.  This topology can be characterized as follows.

Let $K$ denote the set of functions on $\mathbb C^n$ of the form $k(z)=k_1(\|\text{Re}\,\zeta\|)\,k_2(\|\text{Im}\,\zeta\|)$, where $k_1$ and $k_2$ are positive, continuous, and increasing, $k_1$ grows faster than any polynomial, and $k_2$ grows faster than any linear exponential. According to \cite{Ehrenpreis1970}, p.~156, a neighborhood basis of $0$ in $\mathbb E'_n$ consists of the sets 
$$
N_k=\{F\in\mathbb E'_n\;:\;\sup_\zeta |F(\zeta)|/k(\zeta)<1\}
$$
as $k$ runs through $K$.

Now suppose that $r/s$ is not $j_\nu$-rational.  Then $j_\nu(rz)$ and $j_\nu(sz)$ have no common roots.  According to \cite{Schwartz1948}, this implies that the closure of the ideal generated by $j_\nu(rz)$ and $j_\nu(sz)$ in $\mathbb E'_1$ is $\mathbb E'_1$ itself.  Hence there is a net $\{H_\lambda\}_{\lambda\in\Lambda}$, where $\Lambda$ is a directed set, such that $H_\lambda\to 1$ in $\mathbb E'_1$, where
$$
H_\lambda(z)=A_\lambda(z)\,j_\nu(rz)+B_\lambda(z)\,j_\nu(sz),\qquad\text{with }A_\lambda,\;B_\lambda\in \mathbb E'_1,
$$
for all $\lambda\in\Lambda$.  In the above, we can take $A_\lambda$ and $B_\lambda$ to be even functions of $z$.

For any $k\in K$, we conclude that $H_\lambda-1\in N_k$ (in $\mathbb E_1'$) for all sufficiently ``large'' $\lambda$; that is to say
\begin{equation}\label{E:N_k-est}
\sup_{z\in\mathbb C} \frac{|H_\lambda(z)-1|}{k_1(|\text{Re}\,z|)\,k_2(|\text{Im}\,z|)}<1\qquad\text{for all ``large'' }\lambda.
\end{equation}
Now for each $\lambda\in\Lambda$, put $a_\lambda(\zeta)=A_\lambda((\zeta\cdot\zeta)^{1/2}),\; b_\lambda(\zeta)=B_\lambda((\zeta\cdot\zeta)^{1/2})$, and $h_\lambda(\zeta)=H_\lambda((\zeta\cdot\zeta)^{1/2})$.  By \eqref{E:zeta-estimate} and the Paley-Wiener theorem for $\mathscr E'(\mathbb R^n)$, $a_\lambda,\,b_\lambda$ and $h_\lambda$ belong to $\mathbb E'_n$.  Moreover, \eqref{E:N_k-est} and \eqref{E:zeta-estimate} imply that
$$
\sup_{\zeta\in\mathbb C^n} \frac{|h_\lambda(\zeta)-1|}{k_1(\|\text{Re}\,\zeta\|)\,k_2(\|\text{Im}\,\zeta\|)}<1.
$$
Since $k\in K$ is arbitrary, it follows that $h_\lambda\to 1$ in $\mathbb E'_n$.  We conclude that the closed ideal in $\mathbb E_n'$ 
generated by $j_\nu(r(\zeta\cdot\zeta)^{1/2})$ and $j_\nu(s(\zeta\cdot\zeta)^{1/2})$ is $\mathbb E_n'$ itself.

Consequently, the closed ideal in the convolution algebra $\mathscr E'(\mathbb R^n)$ generated by $S_\nu^r$ and $S^s_\nu$ is $\mathscr E'(\mathbb R^n)$ itself.  More precisely, if we let $\Phi_\lambda$ and $\Psi_\lambda$ in $\mathscr E'(\mathbb R^n)$ be the inverse Fourier transforms of $a_\lambda$ and $b_\lambda$, respectively, then
\begin{equation}\label{E:conv-convergence}
S^r_\nu*\Phi_\lambda+S^s_\nu*\Psi_\lambda\to\delta_0.
\end{equation}
Our assumption is that $f*m^r_\alpha=0$ and $f*m^s_\alpha=0$.  By \eqref{E:malphaFT}, this implies that $f*S^r_\nu=0$ and $f*S^s_\nu=0$.  Taking the convolution of both sides of \eqref{E:conv-convergence} with $f$, we obtain $f=0$, proving the theorem.
\end{proof}


Let us now consider the existence question.  Our aim is to prove that if $r/s$ is neither $j_\nu$-rational nor $j_\nu$-Liouville, then for any $g,\,h\in\mathscr E(\mathbb R^n)$ satisfying the compatibility condition \eqref{E:compatibility2}, there is an $f\in\mathscr E(\mathbb R^n)$ such that $f*m^r_\alpha=g$ and $f*m^s_\alpha=h$.  For this, the lemma below will be useful.  Let $Z$ denote the set of zeros of $j_\nu$.

\begin{lemma}\label{T:nearestzero}
Let $r$ and $s$ be positive real numbers, $r\neq s$, and suppose that $r/s$ is neither $j_\nu$-rational nor $j_\nu$-Liouville.  Then there exists a positive constant $C$ and a positive integer $N$ such that 
$$
|rz-a|+|sz-b|\geq C(1+|z|)^{-N}
$$
for all $z\in\mathbb C$ and all $a$ and $b$ in $Z$.
\end{lemma}
\begin{proof} 
We will prove the equivalent assertion that there exists a $C>0$ such that
\begin{equation}\label{E:lower-est}
\left|z-\frac a r\right|+\left|z-\frac b s\right|\geq C\,(1+|z|)^{-N},\qquad z\in\mathbb C,\;a,b\in Z.
\end{equation}
Let us first consider the case when $\nu$ is not real.  In this situation all that will be needed to prove the estimate \eqref{E:lower-est} is the assumption that $r/s$ is not $j_\nu$-rational. In fact we assert here that there is a constant $d>0$ (depending on $r,\,s$ and $\nu$) such that
\begin{equation}\label{E:lower-est1}
\left|z-\frac a r\right|+\left|z-\frac b s\right|\geq d\qquad z\in\mathbb C,\;a,b\in Z.
\end{equation}

Let $A=(1/r)\,Z=\{a/r\,|\,a\in Z\}$ and $B=(1/s)\,Z=\{b/s\,|\,b\in Z\}$.  Since $r/s$ is not $j_\nu$-rational, the sets $A$ and $B$ are disjoint.

Now it is clear from the description \eqref{E:largezeros} of the large elements of $Z$ that the elements of $A$ on the right half plane with large modulus tend towards the horizontal line $\text{Im}\,z=\pi\,\text{Im}\,\nu/(2r)$, 
and are spaced approximately $\pi/r$ units apart.  Similarly the 
elements of $B$ on the right half plane with large modulus  tend towards the horizontal line $\text{Im}\,z=\pi\,\text{Im}\,\nu/(2s)$ and are spaced approximately $\pi/s$ units apart.  Since $r\neq s$ and $\text{Im}\,\nu\neq 0$, these horizontal lines are distinct and can be separated by horizontal bands of small width spaced a positive distance apart.  The elements of $A$ and $B$ with large modulus and with positive real part lie inside these horizontal bands.  

Since $j_\nu$ is even, we can describe  the elements of $A$ and $B$ on the left half plane with large modulus in a similar fashion.

Again by \eqref{E:largezeros} we know that there are only finitely many points in $A$ and $B$ with bounded modulus.  All of this shows that there is a $d>0$ such that
\begin{equation}\label{E:lower-est2}
\left|\frac{a}{r}-\frac{b}{s}\right|\geq d
\end{equation}
for all $a,\,b\in Z$.  This implies \eqref{E:lower-est1}.

We now consider the case when $\nu$ is real.  Here we will need the full assumption that $r/s$ is neither $j_\nu$-rational nor $j_\nu$-Liouville.   According to \cite{WatsonTreatiseBessel}, Sec.~15.2, $j_\nu$ has at most finitely many zeros which are not real (and none when $\nu>-1$). Thus we need only prove the estimate \eqref{E:lower-est} for all $z\in\mathbb C$ and all $a$ and $b$ in $Z\cap\mathbb R$. Furthermore, for all $a$ and $b$ in $Z\cap\mathbb R$, we in fact need only establish \eqref{E:lower-est} all $x\in\mathbb R$.  The estimate \eqref{E:lower-est} will then follow for general $z=x+iy \in \mathbb C$, since
\begin{align*}
\left|z-\frac a r\right|+\left|z-\frac b s\right|&\geq \left|x-\frac a r\right|+\left|x-\frac b s\right|\\
&\geq C(1+|x|)^{-N}\\
&\geq C(1+|z|)^{-N}.
\end{align*}

Now since  $r/s$ is not $j_\nu$-Liouville, there is a (large) positive integer $M$ such that
\begin{equation}\label{E:notliouville}
\left|\frac rs-\frac a b\right|\geq |b|^{-M},
\end{equation}
for all $a,\;b\in Z$, with $|b|>1$.  If $x\in\mathbb R$ this gives
\begin{equation}\label{E:lower-est3}
\left|x-\frac a r\right|+\left|x-\frac b s\right|\geq \left|\frac a r-\frac b s\right|=\frac {|b|}r\,\left|\frac{a}{b}-\frac rs\right|\geq \frac 1r\,|b|^{1-M}.
\end{equation}
The description \eqref{E:largezeros} shows that there is a large $R>0$ such that the elements of $A$ with modulus $>R$ are spaced no more than $2\pi/r$ units apart and the elements of $B$ with modulus $>R$ are spaced no more than $2\pi/s$ units apart.

Fix $x\in\mathbb R$.  We need only prove \eqref{E:lower-est} for $|x|>R$. Let $a/r$ and $b/s$ be points in $A$ and $B$, respectively, with $a$ and $b$ real, that are closest to $x$.   
By our choice of $R$, we have $|x-b/s|\leq 2\pi/s$, so $|b|\leq s\,|x|+2\pi$, and so \eqref{E:lower-est3} gives
$$
\left|x-\frac a r\right|+\left|x-\frac b s\right|\geq \frac{1}{r}\,(2\pi+s\,|x|)^{1-M}.
$$
Putting $N=M-1$ and adjusting the constant $C$, in part to account for the case $|x|<R$, this proves \eqref{E:lower-est}, and the lemma.
\end{proof}

\begin{corollary}\label{T:nu-complex}
Suppose that $\nu$ is not real.  Then there are no $j_\nu$-Liouville numbers.
\end{corollary}
\begin{proof}
We already know that any $j_\nu$-Liouville number must be real.  So suppose that $r$ is a non-zero real number and that $r$ is not $j_\nu$-rational.  We may assume that $r$ is positive.  Note that $r\neq 1$.  The inequality  \eqref{E:lower-est2}, with $s=1$, shows that
$$
\left|r-\frac{a}{b}\right|\geq \frac{dr}{|b|}
$$
for all $a,\,b\in Z$.  (The inequality \eqref{E:lower-est2} only assumes that $r=r/1$ is not $j_\nu$-rational.) From this it is clear that $r$ is not $j_\nu$-Liouville. 
\end{proof}


Below is our main existence result.

\begin{theorem}\label{thm:besselnonLiouville}
  Given $g,h$ in $\mathscr{E}(\mathbb{R}^n)$ and
  $r,s>0$ and assume that $g*m_\alpha^r=h*m_\alpha^s$.
  If $r/s$ is neither
  $j_{\nu}$-rational nor $j_{\nu}$-Liouville,
  then there is a unique $f\in \mathscr{E}(\mathbb{R}^n)$ for which
  $g=f*m_\alpha^s$ and $h=f*m_\alpha^r$.
\end{theorem}
Notice that the uniqueness statement for the case $\nu=0$, in which case
the involved operators are spherical mean value operators, is exactly
the two-radius Pompeiu theorem by Zalcman.

The following result shows why it is important to have ratios which are neither $j_\nu$-rational nor $j_\nu$-Liouville.

\begin{theorem}\label{thm:besselLiouville}
  If $r/s$ is 
  either
  $j_{\nu}$-rational or 
  $j_{\nu}$-Liouville,
  then there are functions $g$ and $h$ in $\mathscr{E}(\mathbb{R}^n)$ satisfying
  $g*m_\alpha^r=h*m_\alpha^s$ yet
  there is no 
  $f\in \mathscr{E}(\mathbb{R}^n)$ for which
  $g=f*m_\alpha^s$ and $h=f*m_\alpha^r$.
\end{theorem}

For the proof of Theorem~\ref{thm:besselnonLiouville}
we will rely on the following result. 
\begin{lemma}
 Let $S,T\in \mathscr{E}'(\mathbb{R}^n)$ and assume there is an $A>0$ such that
   \begin{equation}
     \label{eq:jointslowdecay}
     |\widetilde{S}(\zeta)|+|\widetilde{T}(\zeta)| \geq (A+\|\zeta\|)^{-A}\exp(-A\|\mathrm{Im}(\zeta)\|), \qquad \text{for all $\zeta\in \mathbb{C}^n$}.
 \end{equation}
If $g,h$ are functions in $\mathscr{E}(\mathbb{R}^n)$ such that
$$
g*S=h*T,
$$
then there is a function $f\in\mathscr{E}(\mathbb{R}^n)$ such that
\begin{equation}\label{E:convideal}
g=f*T\qquad\text{and}\qquad h=f*S.
\end{equation}
\end{lemma}
\begin{proof}
    According to Ehrenpreis (\cite{Ehrenpreis1970}, Sec.~ XI.1, the condition \eqref{eq:jointslowdecay} implies that the ideal generated by $S$ and $T$ in the convolution algebra $\mathscr E'(\mathbb R^n)$ is $\mathscr E'(\mathbb R^n)$ itself.  In particular there are $\Phi$ and $\Psi$ in $\mathscr E'$ such that
    $$
    S*\Phi + T*\Psi=\delta_0.
    $$
Putting $f=h*\Phi+g*\Psi$, we obtain \eqref{E:convideal}.  
\end{proof}

(The above result is Lemma~3.17 in \cite{CGKW2024} but we include it here for completeness.)

We also need the following result for Bessel functions.
\begin{lemma}
  Suppose that $r,s>0$ and that $r/s$ is 
  neither $j_\nu$-rational
  nor $j_\nu$-Liouville. Then there is a constant $C>0$ and
  a positive integer $N$ such that
  $$
  |j_\nu(rz)|+|j_\nu(sz)| \geq C(1+|z|)^{-N}
  $$
  for all $z\in\mathbb{C}$.
\end{lemma}

\begin{proof}
  Notice that $rz$ and $sz$ cannot be simultaneous
  zeros of $j_\nu$, so the sum
  $|j_\nu(rz)|+|j_\nu(sz)|$ is bounded below by
  a constant on any set $|z|\leq R$ with $R>0$.
  Therefore
  it is sufficient to consider large $|z|$.
  Moreover, $j_\nu$ is even, so it also suffices to 
  consider $\mathrm{Re}(z)\geq 0$.

  Let $b_m=m\pi+(\nu-1/2)\pi/2$ be the zeros of
  $\cos(z-(\nu+1/2)\pi/2)$ and let $m$ be large enough
  that the open box 
  $B(b_m,\pi/4) = 
  \{z: |\mathrm{Re}(z-b_m)|< \pi/4,|\mathrm{Im}(z-b_m)|< \pi/4 \}$ contains one and only one zero, denoted 
  $a_m$, of $j_\nu$. Outside all boxes $B(b_m,\pi/4)$, the function $|\cos(z-(\nu+1/2)\pi/2)|$ is bounded
  uniformly below and the function
  $|\tan(z-(\nu+1/2)\pi/2)|$ is bounded uniformly above. Therefore the asymptotic expansion \eqref{E:asymptotic-est} gives
  \begin{align*}
  |j_\nu(z)| &=
  \left|
  \left(\frac{2}{\pi z}\right)^{\frac{1}{2}+\nu}\left[
  1-\frac{1}{z}\tan\left(z-\left(\nu+\frac{1}{2}\right)\frac{\pi}{2}\right)+O\left(\frac{1}{z^2}\right)
  \right]\right|\\&\qquad
  \times \left|
  \cos\left(z-\left(\nu+\frac{1}{2}\right)\frac{\pi}{2}\right)\right| \\
  &\geq C|z|^{-1/2-\mathrm{Re}(\nu)}
  \end{align*}
  for an appropriately chosen constant $C>0$ and
  for $|z|$ sufficiently large.
  Therefore the desired estimate holds true if at least one of $rz$ or $sz$ is outside such a box. 

  We now prove the estimate for the situation where
  both $rz$ and $sz$ are in (possibly different) boxes.
  Let us estimate $|j_\nu(z)|$ when $z$ is in the box
  $B(b_m,\pi/4)$. Remember that $a_m$ is the only root of $j_\nu$ in
  this box and that $a_m$ is simple. Therefore
  $\frac{j_\nu(z)}{z-a_m}$ is entire and it does not vanish on
  the box or its boundary. Therefore the minimum
  of $\left|\frac{j_\nu(z)}{z-a_m}\right|$ is
  attained at some point $w$ on the boundary. 
  Since
  $|w-a_m|\leq \pi/\sqrt{2}$ we get
  $\left|\frac{j_\nu(w)}{w-a_m}\right|
  \geq \frac{\sqrt{2}}{\pi} |j_\nu(w)|$.
  From the previous estimate of $j_\nu$ 
  outside the box we get
  $$
  |j_\nu(z)|\geq C|w|^{-1/2-\mathrm{Re}(\nu)}|z-a_m|.
  $$
  Lastly, since $|z|$ is large and $w$ is in a box of fixed size containing $z$
  we have that $|z|\sim |w|$ which implies
  \begin{equation}\label{E:estimateinbox}
  |j_\nu(z)|\geq C|z|^{-1/2-\mathrm{Re}(\nu)}|z-a_m|.    
  \end{equation}

  Assuming $rz$ is in $B(b_m,\pi/4)$
  and $sz$ is in $B(b_n,\pi/4)$, then
  from \eqref{E:estimateinbox} gives
    \begin{align*}
  |j_\nu(rz)|&+|j_\nu(sz)| \\
  &\geq C|sz|^{-1/2-\mathrm{Re}(\nu)}|sz-a_n| + C|rz|^{-1/2-\mathrm{Re}(\nu)}|rz-a_m|\\
  &\geq C\max(r,s)^{-1/2-\mathrm{Re}(\nu)}
  |z|^{-1/2-\mathrm{Re}(\nu)}(|sz-a_n|+|rz-a_m|)    
  \end{align*}
  and Lemma~\ref{T:nearestzero} gives the desired estimate.

\end{proof}

To finish the proof of Theorem~\ref{thm:besselnonLiouville} it remains  to show that the Fourier 
transforms of the
distributions $m_\alpha^r$ and $m_\alpha^s$
satisfy the estimate \eqref{eq:jointslowdecay} when $r/s$ is neither 
$j_\nu$-rational nor $j_\nu$-Liouville.
Choose a $z$ such that $z^2=\zeta\cdot\zeta$. Then $|z|\leq \|\zeta\|$ and we get that
\begin{align*}
|(m_\alpha^r)^\sim(\zeta)|+|(m_\alpha^s)^\sim(\zeta)| 
&= \Gamma(\alpha+n/2)(|j_\nu(r z )|+|j_\nu(s z )|) \\
&\geq C (1+|z|)^{-N} \\
&\geq C(1+\|\zeta\|)^{-N}.
\end{align*}
This proves  Theorem~\ref{thm:besselnonLiouville}.

We now turn to Theorem~\ref{thm:besselLiouville}.
For a compactly supported distribution $\mu$ 
we let $C_\mu$ denote the operator
on $\mathscr{E}(\mathbb{R}^n)$ corresponding to 
convolution by $\mu$.
\begin{lemma}\label{t:eigenspace}
    The space $\mathscr{E}_\lambda=\{ f\in\mathscr{E}(\mathbb{R}^n) : \Delta f = -\lambda^2 f \}$ is an eigenspace for $C_{m_\alpha^t}$. In particular
    $$
    f*m_\alpha^t = \Gamma(\alpha+n/2)j_\nu(\lambda t) f.
    $$
\end{lemma}

\begin{proof}
    Both sides of the equation above are solutions
    to the EPD equation with initial conditions
    $u(x,0)=f(x)$ and 
    $\frac{\partial u}{\partial t} (x,0)=0$.
    For the left hand side this follows from
    Lemma~\ref{T:EPD1} and \eqref{E:init-data}. Taking Fourier transform in the space variable of \eqref{E:EPD-dist} and
    coupling it with \eqref{E:malphaFT} gives
    the relation
    $$
    -(\zeta\cdot\zeta)j_\nu(t(\zeta\cdot\zeta)^{1/2})
    = (\zeta\cdot\zeta) j_\nu''(t(\zeta\cdot\zeta)^{1/2})
    +
    (\zeta\cdot\zeta)^{1/2}\frac{n-1+\alpha}{t} j_\nu'(t(\zeta\cdot\zeta)^{1/2}).
    $$
    Letting $\zeta=\lambda e_1$ this gives
    $$
    -\lambda^2 j_\nu(t\lambda)
    = \lambda^2 j_\nu''(t\lambda)
    +
    \lambda \frac{n-1+\alpha}{t} j_\nu'(t\lambda),
    $$
    which is used to prove that the right hand side 
    is also a solution.
\end{proof}

The following result, while not needed in the sequel, characterizes the space of mean periodic functions with respect to $m_\alpha^t$.  Its proof is analogous to the proof of Theorem~2.5 from
\cite{CGKW2024}, and we shall omit it. 
\begin{theorem}
    The kernel $\ker(m_\alpha^t)$ is the closure 
    in $\mathscr{E}(\mathbb{R}^n)$ of
    $\oplus_{j}\mathscr{E}_{a_j/t}$ where
    $a_j$ denotes the zeros of $j_\nu$.
\end{theorem}

As in Subsection~\ref{sec:Besselprelims}, for any $\alpha\in\mathbb C$, we put $\nu=\alpha+(n-2)/2$.

\begin{lemma}
$C_{m_\alpha^s}$ is injective on $\ker(C_{m_\alpha^r})$ if and only if $r/s$ is
not $j_\nu$-rational.
\end{lemma}
\begin{proof}
    If $r/s=a/b$ 
    where $a,b$ are zeros of $j_\nu$, then due to Lemma~\ref{t:eigenspace} the space
    $\mathscr{E}_{a/r}$ is in both 
    $\ker(C_{m_\alpha^r})$ and 
    $\ker(C_{m_\alpha^s})$.   It follows that $C_{m_\alpha^s}$ is not injective on
    $\ker(C_{m_\alpha^r})$.

    If $r/s$ is not $j_\nu$-rational, then for 
    $f$ in both 
    $\ker(C_{m_\alpha^r})$ and 
    $\ker(C_{m_\alpha^s})$ 
    Theorem~\ref{T:2snapshotuniqueness}
    shows that $f=0$.  
\end{proof}

\begin{lemma}
    Let $r/s$ be a $j_\nu$-Liouville number, then the convolution operator
    $C_{m_\alpha^r}:\ker(C_{m_\alpha^s})\to \ker(C_{m_\alpha^s})$ is not surjective.
\end{lemma}
\begin{proof}
    The proof follows that of Lemma~3.12 in
    \cite{CGKW2024} with obvious modifications.
\end{proof}

We now return to the proof of Theorem~\ref{thm:besselLiouville}.
When $r/s$ is a $j_\nu$-Liouville number 
we can let $g=0$
and let $h$ be a function in $\ker(C_{m_\alpha^s})$
which is not in $C_{m_\alpha^r}(\ker(C_{m_\alpha^s}))$. Then $g$ and $h$
satisfy the compatibility condition
\eqref{E:compatibility2} yet there is no
$f$ for which $h=f*m_\alpha^r$ and
$g=f*m_\alpha^s$.

In the case when $r/s$ is $j_\nu$-rational we write $r/s=a/b$ where $a$ and $b$ are distinct zeros of
$j_\nu$. Define the distribution
$\widetilde{\phi_\alpha}(\zeta) = (\zeta\cdot\zeta)-(a/r)^2
= (\zeta\cdot\zeta)-(b/s)^2$ and
$\Psi_\alpha^r$ by
$(\Psi_\alpha^r)^{\sim}(\zeta)=(m^r_\alpha)^{\sim}(\zeta)/\widetilde{\phi_\alpha}(\zeta)$. It then follows that
$m_\alpha^s*\Psi_\alpha^r = 
m_\alpha^r*\Psi_\alpha^s$. If $u(x,t)=f*m_\alpha^s(x)$ is a solution to the EPD equation, then
we get that
the snapshots $u_s,u_r$ must satisfy
$u_s*\Psi_\alpha^r = u_r*\Psi_\alpha^s$. This leads to the stronger compatibility condition \begin{equation} \label{E:strongcompatibility}
    g*\Psi_\alpha^r=h*\Psi_\alpha^s
\end{equation}
for snapshots $g,h$. 
Let $g=e^{ia/r x_1}$
and $h=0$ then 
both $g$ and $h$ are in the kernel of $C_{\phi_\alpha}$ and therefore
satisfy the original compatibility condition \eqref{E:compatibility2}.
But 
$g*\Psi_\alpha^r(x)=\Gamma(\alpha+n/2)j_\nu'(a)g(x)r^2/(2a)$
and 
$h*\Psi_\alpha^s(x)=0$ and these
cannot be equal since the root $a$ is simple. Thus $g$ and $h$ do not satisfy
the stronger compatibility condition \eqref{E:strongcompatibility} and there is no $f$ for which
$h=f*m_\alpha^r$ and $g=f*m_\alpha^s$. 
This completes the proof of Theorem~\ref{thm:besselLiouville}.

\begin{remark}
  In a forthcoming paper we will investigate when a more restrictive compatibility condition, replacing $g*m_\alpha^r=h*m_\alpha^s$,
  will ensure the existence of $f$
  in the above theorem when $r/s$ is $j_\nu$-rational. A similar problem was addressed in 
  Theorem~3.19 of \cite{CGKW2024}.
\end{remark}

\section{Properties of the set of $j_\nu$-Liouville numbers}\label{Sec:BLproperties}

In this section, we prove that for any $\nu\in\mathbb R$,  the set of $j_\nu$-Liouville numbers 
is uncountable, dense in $\mathbb{R}$, and of measure zero. 
As mentioned in Corollary \ref{T:nu-complex}, the set of $j_\nu$-Liouville numbers is empty if $\mathrm{Im}(\nu)\neq 0$.  We also note that for real $\nu$, all but finitely many zeros of $j_\nu$ are real, and that all zeros are real when $\nu>-1$ (\cite{WatsonTreatiseBessel}, Sec. 15.27).

\subsection{Cardinality and density of the set of $j_\nu$-Liouville numbers}\label{subsec:Cardinality-Bessel-Liouville}

As we stated in \eqref{E:largezeros},
the large zeros of $j_\nu$ with nonnegative real part are indexed by all sufficiently large positive integers $m$ as follows:
$$
a_m = m\pi+\nu\pi/2-\pi/4+\mathcal{O}(m^{-1}), \qquad m\to\infty.
$$
It follows easily from this that there exist positive constants $C_1, \, C_2$ such that
$$
\begin{aligned}
{} &  C_1^{-1} n \leq a_n \leq C_1 n,  \quad (n=1,2,3, \cdots), \\
{} & C_2^{-1} \leq a_{\ell+1} - a_{\ell} \leq C_2,  \quad (\ell=1,2,3, \cdots).
\end{aligned}
$$
Thus we have
\begin{equation}\label{estimate-defference-two-ratios}
\frac{1}{ C_1 C_2 n }  \leq 
   \frac{ a_{\ell+1}  }{ a_n  } - \frac{ a_{\ell }  }{ a_n  }
     \leq \frac{C_1 C_2}{n}.
\qquad \text{for} \; \, \ell, n=1,2,3, \cdots.
\end{equation}
Using the above inequality, we construct a candidate of a $j_\nu$-Liouville number
as a limit of a sequence of ratios of two zeros of the Bessel functions.
For this purpose, we introduce a special ratio $\theta (n, x)$ of two zeros of the Bessel function
as follows.
For $x>0$ and for a positive integer $n$, let us choose $\ell$ so that
$$
\frac{ a_{\ell-1}  }{ a_n  }  < x \leq  \frac{ a_{\ell }  }{ a_n  },
$$
and we define $\theta (n, x)$ by
$$
\theta (n, x) = \frac{ a_{\ell+1}  }{ a_n  } 
$$
Then we see easily that 
there exists a positive constant $C>1$ independent of $n$ and $x$ such that
\begin{equation}\label{theta-x-inequality}
\frac{1}{ C n}  \leq \theta (n, x) -x \leq \frac{C}{n}.
\end{equation}
In fact, the above inequality follows from 
inequality \eqref{estimate-defference-two-ratios} and the following inequality
$$
\frac{ a_{\ell+1}  }{ a_n  } - \frac{ a_{\ell }  }{ a_n  }
\leq \theta (n, x) -x \leq 
\frac{ a_{\ell+1}  }{ a_n  } - \frac{ a_{\ell-1}  }{ a_n  }.
$$
Now let us choose a positive integer $N$ so that
\begin{equation}\label{choice-of-N}
      1+ \frac{1}{ N^{2}  } +\frac{1}{ N^{ 2 \cdot 3 \cdot 4  } }
               + \frac{1}{ N^{ 2 \cdot 3 \cdot 4 \cdot 5 \cdot 6 }  } + \cdots  
               < \min \{ C, \frac{N}{C^2}  \}.
\end{equation}
Next, let $\mathcal{B} := \{0, 1 \}^{\mathbb{N}}$, the set of all sequences 
$\varepsilon = (\varepsilon_1, \varepsilon_2, \cdots )$ with $\varepsilon_j =0$ or $1$.
We note that $\mathcal{B}$ is an uncountable set.
For $\varepsilon  =(\varepsilon_1, \varepsilon_2, \cdots ) \in \mathcal{B}$, 
let 
$$
n_m ( \varepsilon ) = \varepsilon_m N^{(2m-1)!} + (1- \varepsilon_m) N^{(2m)!}.
$$
Let $0<x<y$. We now aim to define a sequence of 
$j_\nu$-rational numbers which is in $(x,y)$
and converges rapidly.
Define the sequence $\{ \Theta_m (\varepsilon) \}_{m=1}^{\infty}$ by
\begin{equation}\label{Theta-sequence}
\begin{cases}
\Theta_1 (\varepsilon)  & = \theta ( n_1 ( \varepsilon ), x ),  \\
\Theta_m (\varepsilon) & 
    = \theta ( n_m ( \varepsilon ), \Theta_{m-1} (\varepsilon) ),
        \quad (m \geq 2).
\end{cases}
\end{equation}
By inequality \eqref{theta-x-inequality}, we have
\begin{equation}\label{Theta-m-difference}
0 < \Theta_{m+1} (\varepsilon) -\Theta_m (\varepsilon)
\leq \frac{C}{ n_{m+1}  (\varepsilon) } \leq \frac{C}{  N^{ (2m+1)! }  },
\qquad (m=1,2,3, \cdots),
\end{equation}
from which we see that $\{ \Theta_m (\varepsilon) \}_{m=1}^{\infty}$ is a Cauchy sequence,
and thus there is a limit
$$
\Theta (\varepsilon) := \lim_{m \to \infty} \, \Theta_m (\varepsilon).
$$
Our next objective is to prove that the mapping 
$\mathcal{B} \ni \varepsilon \mapsto \Theta (\varepsilon) \in \mathbb{R}$ is injective.
For this purpose, we introduce the lexicographical order $\prec$ on $\mathcal{B}$ as follows.
For $\varepsilon  =(\varepsilon_1, \varepsilon_2, \cdots ), \;
\varepsilon'  =(\varepsilon'_1, \varepsilon'_2, \cdots, ) \in \mathcal{B}$, 
we say $\varepsilon \prec \varepsilon'$ if there exists a positive integer $m$ such that
$$
\varepsilon_j = \varepsilon'_j, \quad (j <m), \quad 
\text{and} \; \; \varepsilon_m < \varepsilon'_m.
$$
Then we have 
\begin{proposition}\label{injectivity-of-Theta-map}
If $\varepsilon \prec \varepsilon'$, $\Theta (\varepsilon)  < \Theta (\varepsilon') $.
\end{proposition}
We note that it is easily checked that if 
$\varepsilon \preceq \varepsilon'$, $\Theta (\varepsilon)  \leq \Theta (\varepsilon') $.

\begin{proof}
Take any $\varepsilon  =(\varepsilon_1, \varepsilon_2, \cdots ), \;
\varepsilon'  =(\varepsilon'_1, \varepsilon'_2, \cdots ) \in \mathcal{B}$, 
such that $\varepsilon \prec \varepsilon'$.
Then there exists an $m$ such that
$$
\varepsilon_j = \varepsilon'_j, \quad (j <m), \quad 
\text{and} \; \; \varepsilon_m =0 < \varepsilon'_m=1.
$$
For the above $\varepsilon, \varepsilon'$ and $m$, let
$$
\begin{aligned}
\overline{\varepsilon} & =  (\varepsilon_1, \cdots, \varepsilon_{m-1}, 0, 1, 1, 1, \cdots),\\
\underline{\varepsilon'}  & =  (\varepsilon_1, \cdots, \varepsilon_{m-1}, 1, 0, 0, 0,  \cdots).
\end{aligned}
$$
By definition, we have 
$\varepsilon  \preceq \overline{\varepsilon} \prec \underline{\varepsilon'} \preceq  \varepsilon' $.
So it suffices to prove that 
$\Theta (\overline{\varepsilon} )  < \Theta (\underline{\varepsilon'} ) $.

By inequality \eqref{theta-x-inequality}, we have 
$$
\begin{aligned}
{} & \Theta (\overline{\varepsilon} ) - \Theta_{m-1} (\overline{\varepsilon} ) \\
  & = 
    \{ \Theta_{m} (\overline{\varepsilon} ) - \Theta_{m-1} (\overline{\varepsilon} ) \}
      + \{ \Theta_{m+1} (\overline{\varepsilon} ) - \Theta_{m} (\overline{\varepsilon} ) \}
        +\{ \Theta_{m+2} (\overline{\varepsilon} ) - \Theta_{m+1} (\overline{\varepsilon} ) \} 
         + \cdots  \\
& \leq 
   C \left\{ \frac{1}{ N^{(2m)! } } + \frac{1}{ N^{(2m+1)! }  } 
                       + \frac{1}{ N^{(2m+3)! }  } + \frac{1}{ N^{(2m+5)! }  } +\cdots \right\} \\
& \leq 
   \frac{C}{ N^{(2m)! }  } \times
  \Big\{ 1+ \frac{1}{ N^{(2m+1) -1 } } + \frac{1}{ N^{(2m+1)(2m+2)(2m+3)  -1}    } 
\\ &\quad\quad\quad\quad\quad\quad\quad + 
         \frac{1}{ N^{(2m+1)(2m+2)(2m+3)(2m+4)(2m+5) -1}  } +\cdots \Big\} \\
& \leq 
   \frac{C}{ N^{(2m)!}  } 
      \left\{ 1+ \frac{1}{ N^{2}  } + \frac{1}{ N^{ 2 \cdot 3 \cdot 4  } }
               + \frac{1}{ N^{ 2 \cdot 3 \cdot 4 \cdot 5 \cdot 6 }  } + \cdots \right\}.
\end{aligned}
$$
Again by inequality \eqref{theta-x-inequality}, we have 
$$
\begin{aligned}
{} & \Theta (\underline{\varepsilon'} ) - \Theta_{m-1} (\underline{\varepsilon'} ) \\
  & =
    \{ \Theta_{m} (\underline{\varepsilon'}  ) - \Theta_{m-1} (\underline{\varepsilon'}  ) \}
      + \{ \Theta_{m+1} (\underline{\varepsilon'}  ) - \Theta_{m} (\underline{\varepsilon'}  ) \}
        +\{ \Theta_{m+2} (\underline{\varepsilon'}  ) - \Theta_{m+1} (\underline{\varepsilon'}  ) \} 
         + \cdots  \\
& \geq 
   C^{-1} \left\{ \frac{1}{ N^{(2m-1)! } } + \frac{1}{ N^{(2m+2)! }  } 
                       + \frac{1}{ N^{(2m+4)! }  } + \frac{1}{ N^{(2m+6)! }  } +\cdots \right\} \\
& \geq 
   \frac{1}{ C N^{(2m-1)! }  }.
\end{aligned}
$$
By \eqref{choice-of-N},
we have $\Theta (\overline{\varepsilon} )  < \Theta (\underline{\varepsilon'} ) $.
\end{proof}

\begin{proposition}
For $\varepsilon  \in \mathcal{B}$,
$\Theta (\varepsilon)$ is very rapidly approximated by a sequence of $j_\nu$-rational numbers; i.e. for every $N$ there are two zeros $a,b$ of $j_\nu$ such that 
$$ |\Theta(\varepsilon) - a/b| \leq |b|^N. $$
\end{proposition}

The above result states that $\Theta(\epsilon)$ is either $j_\nu$-rational 
or $j_\nu$-Liouville.

\begin{proof} 
Note that each $\Theta_m (\varepsilon)$ is written as
$$
\Theta_m (\varepsilon) = \frac{ a_{\ell} }{ a_{ n_m( \varepsilon )  }  },
$$
for some $\ell$.
We also note that $a_{ n_m( \varepsilon )  }  \sim \pi n_m( \varepsilon ) $.
In the same manner as in the proof of Proposition \ref{injectivity-of-Theta-map}, 
we have
$$
\begin{aligned}
{}& 0 < \Theta (\varepsilon) - \Theta_m (\varepsilon)  \\
 & = \{ \Theta_{m+1} (\varepsilon) - \Theta_m (\varepsilon) \}
      + \{ \Theta_{m+2} (\varepsilon) - \Theta_{m+1} (\varepsilon) \} 
       + \{ \Theta_{m+3} (\varepsilon) - \Theta_{m+2} (\varepsilon) \} +     \cdots \\
 & \leq 
    \frac{C}{n_{m+1} (\varepsilon) }  +  \frac{C}{n_{m+2} (\varepsilon) } + 
       \frac{C}{n_{m+3} (\varepsilon) }   +     \cdots 
        \qquad ( \text{by } \eqref{Theta-m-difference} )       \\
 & =
   \frac{C}{n_{m+1} (\varepsilon) } 
    \left\{ 1 + \frac{ n_{m+1} (\varepsilon) }{ n_{m+2} (\varepsilon) } 
       + \frac{ n_{m+1} (\varepsilon) }{ n_{m+3} (\varepsilon) }  + \cdots 
       \right\}  \\
 & \leq 
   \frac{C}{n_{m+1} (\varepsilon) } 
    \left\{ 1 + \frac{ 1 }{ N^2 } 
       + \frac{ 1 }{ N^{2 \cdot 3 \cdot 4} } 
    + \frac{1}{ N^{ 2 \cdot 3 \cdot 4 \cdot 5 \cdot 6 }  } + \cdots 
       \right\}  \\
 & \leq 
   \frac{C}{n_{m+1} (\varepsilon) } \times C   \qquad ( \text{by} \eqref{choice-of-N} ) \\
 &  \leq  \frac{C^2}{n_{m} (\varepsilon) ^{2m+1}}  
    \leq \frac{C^3}{a_{ n_m ( \varepsilon ) }^{2m+1}}.
\end{aligned}
$$
The above inequality proves the assertion.
\end{proof}

Let us take the least element ${\bold 0} = (0,0, 0, \cdots )$ 
and the greatest element ${\bold 1} = (1,1, 1, \cdots )$
in $\mathcal{B}$, then for $\varepsilon  \in \mathcal{B}$,
$$
x < \Theta ({\bold 0}) \leq \Theta (\varepsilon) \leq \Theta ({\bold 1}).
$$
Moreover
$$
\begin{aligned}
{}& \Theta ({\bold 1}) 
   < x + C \left\{ 
    \frac{1}{n_1 ({\bold 1})} + \frac{1}{n_2 ({\bold 1})} + \frac{1}{n_3 ({\bold 1})} + \cdots
        \right\} \\
& = x + C \left\{ 
    \frac{1}{N^1} + \frac{1}{ N^{3!} } + \frac{1}{ N^{5!} } + \cdots
        \right\} \\
& < x + \frac{C}{N} \left\{ 
   1 + \frac{1}{ N^{2} } + \frac{1}{ N^{ 2 \cdot 3 \cdot 4} } + \cdots
        \right\}     \qquad ( \text{by} \eqref{choice-of-N} ) \\
&  <    x + \frac{C^2}{N}. 
\end{aligned}        
$$
Let us now choose a sufficiently large $N$ so that $\displaystyle{ x + \frac{C^2}{N} < y }$,
then we have
$$
x < \Theta (\varepsilon) < y,
$$
for any $\varepsilon  \in \mathcal{B}$.

By Proposition \ref{injectivity-of-Theta-map} and by the fact that $\mathcal{B}$ is an uncountable set,
we see that the set $\{ \Theta (\varepsilon) \in \mathbb{R} \, | \, \varepsilon \in \mathcal{B} \, \}$ 
is an uncountable subset of $(x,y)$.
On the other hand, the set of $j_\nu$-rational numbers is countable, so for uncountably many 
$\varepsilon \in \mathcal{B}$, $\Theta (\varepsilon)$ is a $j_\nu$-Liouville number
in $(x,y)$. 
Since $j_\nu$ is even $x$ and $y$ do not need to be positive and we obtain the following.

\begin{theorem}
The set of $j_\nu$-Liouville numbers is uncountable and dense in $\mathbb R$.
\end{theorem}

\subsection{The measure of the set of $j_\nu$-Liouville numbers}\label{subsec:measure-Bessel-Liouville}

In this subsection, we prove that
the Lebesgue measure of
the set of $j_\nu$-Liouville numbers is zero.
For the zeros $\{ a_m \}_{m=1}^{\infty}$ of $j_{\nu}$,
and for positive integers $L$, $n$, and $p$,
let us define four sets 
$\mathcal{Z}_n^{(L)}$, 
$\mathcal{Y}_n^{(L)} (p)$, $\mathcal{Y}^{(L)} (p)$ and $\mathcal{Y}^{(L)}$ by
$$
\begin{aligned}
\mathcal{Z}_n^{(L)} 
& = \left\{ \frac{a_k}{a_n} \in (0, L) \, \Big| \, 1 \leq k < \infty \right\},\\
\mathcal{Y}_n^{(L)} (p) 
 & =  \left\{ y  \in (0, L)  \, \Big| \,  \mathrm{dist} (y, \mathcal{Z}_n^{(L)}) \leq \frac{1}{\alpha_n^p} \right\}, \\
\mathcal{Y}^{(L)} (p) & = \bigcup_{n=1}^{\infty} \mathcal{Y}_n^{(L)} (p),
\quad 
\mathcal{Y}^{(L)} = \bigcap_{p=1}^{\infty} \mathcal{Y}^{(L)} (p).
\end{aligned}
$$
By the above definition, we easily see that 
$y \in \mathcal{Y}^{(L)}$ if and only if
$y \in (0, L)$ and
there exist increasing sequences $\{ n_p \}_{p=1}^{\infty}$ and $\{ \ell_p \}_{p=1}^{\infty}$
such that
$$
\left| \, y - \frac{a_{\ell_p} }{a_{n_p}} \, \right| \leq \frac{1}{a_{n_p}^{\, p}}.
$$
Namely if $y$ is a $j_\nu$-Liouville number in the interval $(0, L)$
then $y \in \mathcal{Y}^{(L)}$.
Let
$$
K(L, n) = \max \left\{ \; k \, \Big| \; \frac{a_k}{a_n} < L \right\}.
$$
Taking \eqref{E:largezeros} into account, we see that there exists
a positive constant $C_3$ independent of $L$ and $n$ such that $K(L, n) < C_3 Ln$.

Let us now prove that for any $L$ the Lebesgue measure of $\mathcal{Y}^{(L)}$ is zero.
By the definition of $\mathcal{Y}_n^{(L)} (p) $,
we have
$$
\mu (\mathcal{Y}_n^{(L)} (p) ) 
  \leq \sum_{k=1}^{K(L, n)} 
    \mu \Big( \Big[ \frac{a_k}{a_n} - \frac{1}{a_n^p} , 
       \frac{a_k}{a_n} + \frac{1}{a_n^p} \Big] \Big) 
  \leq  \frac{ 2 C_3 L n }{ a_n^p },
$$
where $\mu$ denotes the Lebesgue measure on $\mathbb{R}$.
Thus we have
$$
\mu ( \mathcal{Y}^{(L)} (p) ) 
\leq \sum_{n=1}^{\infty} \mu ( \mathcal{Y}_n^{(L)} (p)  ) 
\leq  2 C_3 L \sum_{n=1}^{\infty} \frac{n}{a_n^p}.
$$
Again by \eqref{E:largezeros}, 
the infinite sum in the right hand side of the above inequality converges if $p \geq 3$.
In addition, if we take account of the fact that $1 < \alpha_1 <\alpha_2 < \cdots$,
we see that
$$
\sum_{n=1}^{\infty} \frac{n}{a_n^p} \to 0, \qquad \text{as} \; p \to \infty.
$$
Therefore, for any $L$, we have
$$
\mu (\mathcal{Y}^{(L)} ) = \lim_{p \to \infty} \mu ( \mathcal{Y}^{(L)} (p) ) =0.
$$
Since $\displaystyle{  \mathcal{Y} := \bigcup_{L=1}^{\infty} \mathcal{Y}^{(L)} }$
includes the set of $j_\nu$-Liouville numbers,
we obtain the following.
\begin{theorem}
The Lebesgue measure of the set of $j_\nu$-Liouville numbers is zero.
\end{theorem}

\subsection{The $j_\nu$-Liouville numbers for varying $\nu$}
The following result implies
that there is a $\nu$ for which
the $j_\nu$-Liouville numbers are not the same as classical Liouville numbers, so our work
indeed expands the theory of Liouville numbers.

\begin{theorem}
    Given $\nu$ there is a $\nu'$ such that the
    set of $j_\nu$-Liouville numbers is different from the set of $j_{\nu'}$-Liouville numbers.
\end{theorem}

\begin{proof}
Let $a_{\nu,k}$ refer
to the $k$th positive zero of $j_\nu$.
From \S 15.6 in \cite{WatsonTreatiseBessel} it is known that $a_{\nu,k}$ depends continuously on $\nu$ and therefore the function
$f(\nu)=a_{\nu,1}/a_{\nu,2}$ is also continuous in $\nu$. Moreover, the values of $f$ are all real (see \S 15.25 in \cite{WatsonTreatiseBessel}).
We know that $a_{1/2,k}=k\pi$ and 
from \S 15.32 in \cite{WatsonTreatiseBessel} we also know that $a_{0,1}$ is in
$(7\pi/4,15\pi/8)$ and $a_{0,2}$ is
in $(11\pi/4,23\pi/8)$. Therefore
$f(0) > \frac{7\pi/4}{23\pi/8} = 14/23>1/2$
while $f(1/2) = 1/2.$ For any given 
$\nu$ let $x$ be a $j_{\nu}$-Liouville number
between $f(0)$ and $f(1/2)$.
By the intermediate value theorem there is a 
$\nu'$ in $(0,1/2)$ for which $f(\nu')=x$. But $f(\nu')$
is $j_{\nu'}$-rational, which finishes the proof.
\end{proof}

\bibliographystyle{alpha}

\bibliography{Gonzalez}

\end{document}